\documentclass[12pt]{amsart}

\usepackage[english]{babel}
\usepackage[utf8x]{inputenc}
\usepackage{geometry}
\usepackage{graphicx}				
\usepackage[colorlinks=true,linkcolor=blue,urlcolor=blue,citecolor=blue]{hyperref}

\usepackage{mathrsfs}
\usepackage{amssymb}
\usepackage{amsthm}
\usepackage{amsmath,amsfonts}
\usepackage{comment,enumitem}
\usepackage{color}

  \newcommand{\Z}{{\mathbb Z}}
  
  \newcommand{\N}{{\mathbb N}}

 \newtheorem{thm}{Theorem}
\newtheorem{prop}[thm]{Proposition}

\newtheorem{lem}[thm]{Lemma}

\theoremstyle{definition}

\newtheorem{remark}[thm]{Remark}

\newtheorem{example}[thm]{Example}

\newtheorem{fact}{Claim}

\makeatletter
\@namedef{subjclassname@2020}{%
  \textup{2020} Mathematics Subject Classification}
\makeatother

\begin{document}

\title{The least doubling constant of a path graph}

\author[E. Durand-Cartagena]{Estibalitz Durand-Cartagena$^*$}
\address{Departamento de Matem\'atica Aplicada, ETSI Industriales, UNED\\
28040 Madrid, Spain.} 
\email{edurand@ind.uned.es }

\author[J. Soria]{Javier Soria$^\dagger$}
\address{
Instituto de Matem\'atica Interdisciplinar (IMI); Departamento de An\'alisis Ma\-te\-m\'a\-ti\-co y Matem\'atica Aplicada, Universidad Complutense de Madrid\\
28040 Madrid, Spain.}
\email{javier.soria@ucm.es}

\author[P. Tradacete]{Pedro Tradacete$^{\ast\ast}$}
\address{Instituto de Ciencias Matem\'aticas (CSIC-UAM-UC3M-UCM)\\
Consejo Superior de Investigaciones Cient\'ificas\\
C/ Nicol\'as Cabrera, 13--15, Campus de Cantoblanco UAM\\
28049 Madrid, Spain.}
\email{pedro.tradacete@icmat.es}

\thanks{$^*$E. Durand-Cartagena was partially supported by MINECO (Spain), project PGC2018-097286-B-I00 and  the grant 2021-MAT11 (ETSI Industriales, UNED)}

\thanks{$^\dagger$J. Soria  was partially supported by grants PID2020-113048GB-I00 funded by MCIN/AEI/ 10.13039/501100011033, and Grupo UCM-970966.}

\thanks{$^{\ast\ast}$P. Tradacete was partially supported by grants PID2020-116398GB-I00, MTM2016-76808-P, MTM2016-75196-P and CEX2019-000904-S funded by MCIN/AEI/ 10.13039/501100011033.}

\begin{abstract}
We study the least doubling constant $C_G$ among all possible doubling measures defined on a path graph $G$. We consider both finite and infinite cases and show that, if $G=\mathbb Z$,  $C_{\mathbb Z}=3$, while for $G=L_n$, the path graph with $n$ vertices, one has $1+2\cos(\frac{\pi}{n+1})\leq C_{L_n}<3$, with equality on the lower bound if and only if $n\le8$. Moreover, we analyze the structure of doubling minimizers on $L_n$ and $\Z$, those measures whose doubling constant is the smallest possible.
\end{abstract}


\date{\today}

\maketitle


\section{Introduction}

This paper is motivated by the study of doubling measures on graphs, already started in \cite{ST}. In order to properly explain our approach, we will first recall some terminology from the abstract theory of measure metric spaces: Given a metric space $(X,d)$, a Borel measure $\mu$ on $X$ is said to be a \emph{doubling measure} whenever
\begin{equation}\label{defdb}
C_\mu=\sup_{x\in X, r>0}\frac{\mu\big(B(x,2r)\big)}{\mu\big(B(x,r)\big)}<\infty,
\end{equation}
where $B(x,r)=\{y \in X: d(x,y)<r\}$ denotes the open ball of center $x$ and with radius $r$. The number $C_\mu$ is called the \emph{doubling constant} of the measure $\mu$. Doubling measures play a fundamental role in the extension of classical results in analysis on Euclidean spaces to a more general setting (cf. \cite{CW, MaSe, heinonen, HKST} for background and some recent developments). 

Associated to a metric space $(X,d)$, the following invariant was introduced in \cite{ST}
$$
C_{(X,d)}=\inf\{C_\mu: \mu \text{ doubling measure on } X\},
$$
which will be referred to as the \emph{least doubling constant} of $(X,d)$. It was shown in \cite{ST} that, if $X$ supports a doubling measure and contains more than one point, then $C_{(X,d)}\geq 2$. This invariant is somehow related to metric dimension theory and our purpose in this note is to study its properties in the context of graphs. We will only deal with connected simple graphs (without loops nor multiple edges). 

A connected simple graph $G$ with vertex set $V_G$ and edges $E_G$ can be considered as a metric space with the path distance: Given vertices $x,y\in V_G$, a path joining $x$ to $y$ is a collection of edges of the form $\{\{x_{i-1},x_i\}\}_{i=1}^k\subset E_G$ with $x_0=x$ and $x_k=y$. In this case, we say that the path has length $k$. Thus, for $x,y\in V_G$, the distance $d_G(x,y)$ is defined as the smallest possible length of a path joining $x$ to $y$ (see \cite{BM} for standard notations and basic results for graphs).

A positive  measure $\mu$ on a graph $G$ will always be determined by a weight function $\mu:V_G\rightarrow (0,\infty)$, and despite the slight abuse of terminology we will write, for any set $A\subset V_G$,  $\mu(A)=\sum_{v\in A}\mu(v)$. Note that if $V_G$ is finite, every measure on $G$ is doubling for an appropriate constant. However, it is worth to observe that there are graphs which do not support any doubling measures, like the infinite $k$-homogeneous tree $T_k$ (see \cite{NaorTao} and also \cite{sortra} for further examples).

For simplicity, we will denote $C_G=C_{(V_G,d_G)}$. In \cite{ST}, this invariant was computed for certain families of finite graphs. Namely, let $K_n$ denote the complete graph with $n$ vertices, $S_n$ be the star-graph with $n$ vertices (one vertex of degree $n-1$ and $n-1$ leaves), and let $C_n$ be the cycle-graph with $n$ vertices: Then, for $n\geq 3$, we have
\begin{enumerate}
\item[(i)] $C_{C_n}=3$,
\item[(ii)] $C_{S_n}=1+\sqrt{n-1}$,
\item[(iii)] $C_{K_n}=n$.
\end{enumerate}

Our main goal in this work is to enlarge this list by estimating the least doubling constant of path or linear graphs $L_n$, $\mathbb N$ and $\mathbb Z$, showing that, somehow surprisingly, the situation strongly differs from the previous similar cases. Here, for $n\in \mathbb N$, the path $L_n$ can be described as the graph with vertex set $V_{L_n}=\{1,\ldots,n\}$ and edge set $E_{L_n}=\{\{j,j+1\}:1\leq j\leq n-1\}$. Similarly, we can view $\mathbb N$ and $\mathbb Z$ as infinite path graphs with $V_{\mathbb N}=\mathbb N$, $V_{\mathbb Z}=\mathbb Z$, $E_{\mathbb N}=\{\{j,j+1\}:j\in\mathbb N\}$ and $E_{\mathbb Z}=\{\{j,j+1\}:j\in\mathbb Z\}$. In particular, we will show that $C_{\mathbb N}=C_{\mathbb Z}=3=\lim_{n\to\infty} C_{L_n}$. In the finite case, we will actually prove, in Theorem~\ref{lessthan3}, that 
$$
1+2\cos\Big(\frac{\pi}{n+1}\Big)\leq C_{L_n}<  3,
$$ 
with equality on the left hand side inequality only for  $2\le n\le 8$. 
\medskip

The paper is organized as follows: In Section~\ref{general} we prove, in Proposition~\ref{p:DMconvex}, that the set of doubling minimizers is always a non-empty convex cone in the set of measures over a graph. In Section~\ref{sec3} we obtain, in Theorem~\ref{thm:C^0_Ln}, the exact value of the auxiliary constant $$C_{L_n}^0=\sup_{x\in V_G}\frac{\mu(B(x,1))}{\mu(x)},$$ which provides the lower bound for $C_{L_n}$. Section~\ref{constz} is devoted to the case of the infinite linear graph, $\Z$, and we determine, in Theorem~\ref{const3}, that  $C_{\mathbb Z}=C_{\mathbb Z}^0=3$, as well as the uniqueness of the minimizers given by the counting measure (see Theorem~\ref{t:minimizersZ}).  In Section~\ref{secc5} we prove  in Theorem~\ref{t:M1M2M3} that when computing the least doubling constant $C_{L_n}$ it is not necessary to consider all possible quotients of measures of balls, thus reducing the complexity of the problem to some  simpler cases. As a consequence of this, we deduce that the sequence of $C_{L_n}$ is non-decreasing in $n\in\mathbb N$ (Proposition \ref{p:C_Ln non-decreasing}). In Section~\ref{secct6} we consider the most important properties for the doubling minimizers on $L_n$ and prove, in Theorem~\ref{t:mainDM}, the fundamental formula relating  $C_\mu$ to some very concrete  quotients, one of the basic results for these measures. We finish, in Section~\ref{finalrem}, with some final remarks, comments and open questions which we consider to be of interest for future developments.

\section{Doubling minimizers on  graphs}\label{general}

Given a graph $G$, let $\operatorname{diam}(G)=\sup\{d_G(x,y):x,y\in V_G\}$. It is easy to see that
\begin{equation}\label{eq:C_G}
C_G=\inf_{\mu}\sup\Big\{\frac{\mu(B(x,2k+1))}{\mu(B(x,k))}: x\in V_G,\,k\in\mathbb Z, \,0\leq k\leq \Big\lceil \frac{\operatorname{diam}(G)-1}{2}\Big\rceil \Big\},
\end{equation}
the infimum being taken over all doubling measures $\mu$ on $G$ (if $\operatorname{diam}(G)=\infty$, then $k\in\mathbb N\cup\{0\}$). Here, and throughout the rest of paper, $B(x,r)=\{y:d(x,y)\leq r\}$ denotes the closed ball with center $x$ and radius $r$ (in the discrete setting, we can equivalently work with either open or closed balls), and $\lceil s\rceil=\min\{n\in\mathbb Z:s\leq n\}$.

If the infimum in \eqref{eq:C_G} is attained at some $\mu$ (which is the case in particular for finite graphs), such a measure will be called a \emph{doubling minimizer} for $G$. Given a graph $G$, with $C_G<\infty$, let us denote the set of doubling minimizers by 
$$
DM(G)=\{\mu:\text{ doubling measure on }G \text{ and } C_\mu=C_{G}\}.
$$
We will see next that the set $DM(G)$ is always a non-empty convex set. Before showing this result, the following elementary lemma will be convenient throughout.

\begin{lem}\label{l:holder}
Let $(\alpha_j)_{j=1}^m, (\beta_j)_{j=1}^m$ be positive real scalars. We have that
$$
\frac{\sum_{j=1}^m\alpha_j}{\sum_{j=1}^m\beta_j}\leq\max_{1\leq j\leq m}\Big\{\frac{\alpha_j}{\beta_j}\Big\}.
$$
Moreover, equality holds if and only if $\frac{\alpha_i}{\beta_i}=\frac{\alpha_j}{\beta_j}$, for every $1\leq i,j\leq m$.
\end{lem}

\begin{proof}
Since $\beta_j>0$, for every $1\leq j\leq m$, the inequality follows from Holder's inequality
$$
\sum_{j=1}^m \alpha_j=\sum_{j=1}^m \frac{\alpha_j}{\beta_j}\beta_j\leq \max_{1\leq j\leq m}\Big\{\frac{\alpha_j}{\beta_j}\Big\}\sum_{j=1}^m\beta_j.
$$

In order to characterize equality in the above expression, one can proceed by induction. Suppose $\frac{\sum_{j=1}^m\alpha_j}{\sum_{j=1}^m\beta_j}=\max_{1\leq j\leq m}\Big\{\frac{\alpha_j}{\beta_j}\Big\},$ and let $j_0$ be an index such that $\frac{\alpha_{j_0}}{\beta_{j_0}}=\max_{1\leq j\leq m}\Big\{\frac{\alpha_j}{\beta_j}\Big\}.$ A simple algebraic manipulation yields that
$$
\frac{\alpha_{j_0}+\sum_{j\neq j_0}\alpha_j}{\beta_{j_0}+\sum_{j\neq j_0}\beta_j}=\frac{\alpha_{j_0}}{\beta_{j_0}}\Leftrightarrow \frac{\sum_{j\neq j_0}\alpha_j}{\sum_{j\neq j_0}\beta_j}=\frac{\alpha_{j_0}}{\beta_{j_0}}.
$$
Hence, we have that
$$
\max_{1\leq j\leq m}\Big\{\frac{\alpha_j}{\beta_j}\Big\}= \frac{\sum_{j\neq j_0}\alpha_j}{\sum_{j\neq j_0}\beta_j}\leq\max_{j\neq j_0}\Big\{\frac{\alpha_j}{\beta_j}\Big\}\leq \max_{1\leq j\leq m}\Big\{\frac{\alpha_j}{\beta_j}\Big\}.
$$
This means in particular that 
$$
 \frac{\sum_{j\neq j_0}\alpha_j}{\sum_{j\neq j_0}\beta_j}=\max_{j\neq j_0}\Big\{\frac{\alpha_j}{\beta_j}\Big\},
$$
and we can conclude by induction.
\end{proof}

\begin{prop}\label{p:DMconvex}
If  $G$ is a graph, with $C_G<\infty$, then $DM(G)$ is a non-empty convex cone.
\end{prop}

\begin{proof}
We show first that $DM(G)\neq \emptyset$. For every $n\in\mathbb N$, let $\nu_n$ be a measure on $G$ such that $C_{\nu_n}\leq C_G+\frac1n$. Fix $v_0\in V_G$ and set $\mu_n=\frac{\nu_n}{\nu_n(v_0)}$. Thus, we have $C_{\nu_n}=C_{\mu_n}\leq C_G+\frac1n$ and $\mu_n(v_0)=1$, for every $n\in\mathbb N$.

We claim that, for every $v\in V_G$, we have that  $\sup_n \mu_n(v)<\infty$. Indeed, this will follow by induction on $m=d(v,v_0)$: clearly, the statement holds for $m=0$; that is, when $v=v_0$; now suppose $\sup_n \mu_n(v)<\infty$, for every $v\in V_G$, with $d(v,v_0)\leq m$, and let $w\in V_G$ with $d(w,v_0)=m+1$; we can pick $v_1\in V_G$ with $d(v_1,v_0)=m$ and $d(w,v_1)=1$; hence, we have
$$
\mu_n(w)\leq \mu_n(B(v_1,1))\leq C_G\,\mu_n(v_1),
$$
and the claim follows.

Let now $\mathcal U$ be a free ultrafilter on $\mathbb N$, and define $\mu$ on $G$ by
$$
\mu(v)=\lim_{n\in\mathcal U}\mu_n(v),
$$
where $\lim_{n\in\mathcal U}$ denotes the limit along the ultrafilter. Note that $\mu(v)$ is well defined because of the claim. It is straightforward to check now that $C_\mu=C_G$.

\medskip
Another elementary proof of this fact, in the case $G$ is finite, goes as follows:  Let $\lambda$ be the counting measure and  $n=\lambda(V_G)$. Then we have the trivial estimate:
$$
C_{G}\leq C_{\lambda}\leq n.
$$
Note also that for $\alpha>0$, $C_{\alpha\mu}=C_\mu$, where $\alpha\mu$ is the measure given by scalar multiplication ($[\alpha\mu](j)=\alpha \cdot \mu(j)$ for every $j\in V_G$). 
Therefore, we can assume without loss of generality that the measures we deal with, when computing the infimum in $C_G$, satisfy $\mu(V_G)=1$ and $C_\mu\leq n$. Thus, for any such measure  and any $j\in V_G$, we have, iterating \eqref{defdb}:
$$
1=\mu(V_G)=\mu(B(j,n))\leq C_\mu^{1+\log_2 n}\mu(j)\leq n^{1+\log_2 n}\mu(j),
$$
which yields
$$
n^{-1-\log_2 n}\leq \min_j \mu(j).
$$
We also have that $\max_j \mu(j)\leq \mu(V_G)=1$. Hence, if we consider the following compact set (as a subset of $\mathbb R^n$, with the Euclidean topology):
$$
K=\{\mu:n^{-1-\log_2 n}\leq \mu(j)\leq 1, \,\text{for } j\in V_G\},
$$
then we have 
$$
C_{G}=\inf_{\mu\in K} C_\mu.
$$
Finally, note that the map 
$$
\begin{array}{ccl}
 K & \longrightarrow  & \mathbb R  \\
 \mu&\longmapsto & C_\mu=\sup\{\frac{\mu(B(j,2k+1))}{\mu(B(j,k))}:j\in V_G, 1\leq k\leq n-1\} 
\end{array}
$$
is a continuous function, and we conclude that there is $\mu\in K$ such that $C_{G}=C_\mu$, as claimed.

\medskip

Let us now study the convexity property. In order to see that  $DM(G)$ is a convex cone, it is enough to show that if $\mu_1,\mu_2\in DM(G)$, then $\mu_1+\mu_2\in DM(G)$.  Thus, suppose $C_{\mu_i}=C_G$, for $i=1,2$ and let $\mu=\mu_1+\mu_2$. For every $j\in V_G$ and $k\in\{0,1,2,\dots\}$, and using Lemma~\ref{l:holder}, we have
\begin{align*}
\frac{\mu(B(j,2k+1))}{\mu(B(j,k))}&=\frac{\mu_1(B(j,2k+1))+\mu_2(B(j,2k+1))}{\mu_1(B(j,k))+\mu_2(B(j,k))}\\
&\leq \max\Big\{\frac{\mu_1(B(j,2k+1))}{\mu_1(B(j,k))},\frac{\mu_2(B(j,2k+1))}{\mu_2(B(j,k))}\Big\}\\
&\leq \max\{C_{\mu_1},C_{\mu_2}\}=C_{G}.
\end{align*}
It then follows  that $C_\mu\leq C_G$. By definition of $C_G$ the conclusion follows.
\end{proof}

\begin{remark}
There is yet another proof of Proposition~\ref{p:DMconvex} based on a  diagonalization argument, avoiding the use of ultrafilters (and, hence, the Axiom of Choice). Regarding the geometry of $DM(G)$, we will show, in Theorem~\ref{t:minimizersZ}, a uniqueness result for the minimizers of  $G=\Z$; namely,   $DM(\Z)$ is the ray generated by the counting measure.
\end{remark}

\section{The constant $C_{L_n}^0$ for finite linear graphs }\label{sec3}

An important information, related to $C_G$, is the consideration of the doubling constant when restricted to the collection of neighbors of a given vertex. In particular, for a doubling measure $\mu$ of $G$,
we define
$$
C_{\mu}^0=\sup_{x\in V_G}\frac{\mu(B(x,1))}{\mu(x)},
$$
and 
$$
C_{G}^0=\inf_{\mu} C_\mu^0.
$$
Observe that 
$$
C_{G}^0\le C_G.
$$
Similar to the result in Proposition~\ref{p:DMconvex}, it can be proved that the constant $C_{G}^0$, which, in general, is easier to calculate, is always attained for some particular doubling measure $\mu$.  In fact, we are now going to find the exact value of the constant $C_{L_n}^0.$ The proof will be split in some elementary lemmas. Throughout this section, for a doubling measure $\mu$ on $L_n$, we set $a_j=\mu(j)$. 

We first observe that, when computing the infimum defining  $C_{L_n}$ or $C_{L_n}^0$, the following result allows us to assume that the measure is symmetric; that is, $a_j=a_{n+1-j}$. Recall that, by assumption, we have $a_j>0$, for every $j$. 

\begin{lem}\label{l:symmetric}
Given $\mu$ on $L_n$, let $\tilde \mu(j)=\mu(j)+\mu(n+1-j)$. Then, $C^0_{\tilde\mu}\leq C^0_\mu$ and $C_{\tilde\mu}\leq C_\mu$.
\end{lem}

\begin{proof}
Given $0\leq k\leq \lceil \frac{n-1}{2}\rceil$, for every $1\leq j\leq n$, using Lemma~\ref{l:holder} we have
\begin{align*}
\frac{\tilde\mu(B(j,2k+1))}{\tilde\mu(B(j,k))}&=\frac{\mu(B(j,2k+1))+\mu(B(n+1-j,2k+1))}{\mu(B(j,k))+\mu(B(n+1-j,k))}\\
&\leq\max\Big\{\frac{\mu(B(j,2k+1))}{\mu(B(j,k))},\frac{\mu(B(n+1-j,2k+1))}{\mu(B(n+1-j,k))}\Big\}\leq C_\mu.
\end{align*}
Hence, $C_{\tilde\mu}\leq C_\mu$. Setting $k=0$, in the above inequality, we also get $C^0_{\tilde\mu}\leq C^0_\mu$ .
\end{proof}

We say that a measure $\mu$ on $L_n$ has a local minimum if there is $1< i< n$ such that $a_i\leq\min\{a_{i-1},a_{i+1}\}$.

\begin{lem}\label{l:localmin}
If a measure $\mu$ on $L_n$ has a local minimum, then $C^0_\mu\geq3$.
\end{lem}

\begin{proof}
Let $1< i< n$ be such that $a_i\leq\min\{a_{i-1},a_{i+1}\}$. Clearly,
$$
C_\mu^0\geq\frac{\mu(B(i,1))}{\mu(B(i,0))}=\frac{a_{i-1}+a_i+a_{i+1}}{a_i}\geq3.
$$
\end{proof}

The following particular measure will play a key role in our computations: 

\begin{lem}\label{l:regularpolygon}
Given $n\in\mathbb N$, let $\sigma$ on $L_n$ be given by 
$$
\sigma(j)=\sin\Big(\frac{j\pi}{n+1}\Big).
$$
Then,  $C^0_\sigma=1+2\cos(\frac{\pi}{n+1}).$
\end{lem}

\begin{proof}
Recall the well-known formula:
\begin{equation}\label{eq:sine}
\sin(x+y)+\sin(x-y)=2\sin x\cos y.
\end{equation}

Note that, for $1\leq j\leq n$, using \eqref{eq:sine}, we have
$$
\frac{\sin\Big(\frac{(j-1)\pi}{n+1}\Big)+\sin\Big(\frac{(j+1)\pi}{n+1}\Big)}{\sin\Big(\frac{j\pi}{n+1}\Big)}=2\cos\Big(\frac{\pi}{n+1}\Big).
$$
Therefore, we have that 
\begin{align*}
C_\sigma^0&=\sup_{1\leq j\leq n} \frac{\sigma(B(j,1))}{\sigma(B(j,0))}\\
&=\sup_{1\leq j\leq n}\frac{\sin\Big(\frac{(j-1)\pi}{n+1}\Big)+\sin\Big(\frac{j\pi}{n+1}\Big)+\sin\Big(\frac{(j+1)\pi}{n+1}\Big)}{\sin\Big(\frac{j\pi}{n+1}\Big)} \\
&=1+2\cos\Big(\frac{\pi}{n+1}\Big).
\end{align*}
\end{proof}

Note that the counting measure $\lambda$ (or equivalently $a_j=1$, for every $j$), satisfies $C_{\lambda}=C_\lambda^0=3$ in $L_n$, so that when computing $C^0_{L_n}$  we can restrict the infimum to the set of measures with $C^0_\mu <3$, which is actually non-empty by Lemma \ref{l:regularpolygon}.

\begin{lem}\label{l:C<3}
If $\mu$ is a symmetric measure in $L_n$, with $C^0_\mu<3$, then $a_i< a_j<\frac{j}{i}a_i$ for $1\leq i<j\leq \lceil\frac{n}{2}\rceil$. 
\end{lem}

\begin{proof}
Using the symmetry and  Lemma~\ref{l:localmin}, we already know that $a_i\leq a_j$. Suppose that for some  $1\leq j<\lceil\frac{n}{2}\rceil$ we had $a_{j-1}=a_{j}$. Then 
$$
C^0_\mu\geq\frac{\mu(B(j,1))}{\mu(B(j,0))}=\frac{a_{j-1}+a_{j}+a_{j+1}}{a_{j}}\geq 3.
$$
Similarly, for $j=\lceil\frac{n}{2}\rceil$, note that $a_{j+1}$ coincides either with $a_j$ or $a_{j-1}$ (depending on whether $n$ is even or odd), so if we had $a_{j-1}=a_j$ in this case, it also follows that
$$
C^0_\mu\geq\frac{\mu(B(j,1))}{\mu(B(j,0))}=\frac{a_{j-1}+a_{j}+a_{j+1}}{a_{j}}= 3.
$$
For the other inequality, we will prove by induction that $a_{j+1}<\frac{j+1}{j} a_j$ for $j<\lceil\frac{n}{2}\rceil$: For $j=1$, if we had $a_2\geq 2a_1$, then
$$
C^0_\mu\geq\frac{\mu(B(1,1))}{\mu(B(1,0))}=\frac{a_1+a_2}{a_1}\geq 3.
$$
Now, suppose $a_j<\frac{j}{j-1}a_{j-1}$. If we had $a_{j+1}\geq\frac{j+1}ja_j$, then 
$$
C^0_\mu\geq\frac{\mu(B(j,1))}{\mu(B(j,0))}=\frac{a_{j-1}+a_j+a_{j+1}}{a_j}> \frac{j-1}{j}+1+\frac{j+1}j =3.
$$
In particular, it follows that for $1\leq i<j\leq\lceil\frac{n}{2}\rceil$ we have $a_j<\frac{j}{i}a_i$.
\end{proof}

\begin{lem}\label{l:lowerbound}
Let $n\in\mathbb N$, $a_1,\ldots, a_n\in\mathbb R_+$, such that $a_i=a_{n+1-i}$ and $a_i\leq a_{i+1}$, for $1\leq i<\lceil \frac{n}{2}\rceil$. Setting $a_0=a_{n+1}=0$ we have
$$
\max_{1\leq i\leq n}\Big\{\frac{a_{i-1}+a_{i+1}}{a_i}\Big\}\geq2\cos\Big(\frac{\pi}{n+1}\Big).
$$
\end{lem}

\begin{proof}
Let $m=\lceil \frac{n}{2}\rceil$. Dividing all the numbers $a_i$ by the largest of them, that is $a_m$, we can take an increasing family $(\alpha_i)_{i=1}^m\subset[0,\frac{\pi}{2}]$ such that for $1\leq i\leq m$
$$
a_i=\sin \alpha_i.
$$
Note that $\alpha_m=\frac{\pi}{2}$. We set $\alpha_0=0=\alpha_{n+1}$, and for $1\leq i\leq m$, let $d_i=\alpha_i-\alpha_{i-1}$. Let $1\leq i_0\leq m$ such that 
$$
d_{i_0}=\min_{1\leq i\leq m}d_{i}.
$$
Note that, in particular,  since $\sum_{i=1}^{n+1}d_i=\pi$, we must have $d_{i_0}\leq\frac{\pi}{n+1}$.
\medskip

Now, we distinguish three cases:
\begin{enumerate}[leftmargin=*]
\item[(a)] Suppose that $d_{i_0+1}\leq\frac{\pi}{n+1}$.  Hence, using the trigonometric formula for the sum of angles, the fact that $\cos x$ is a decreasing function for $x\in[0,\pi]$ and the choice of $i_0$ we have that
\begin{align*}
\max_{1\leq i\leq n}\Big\{\frac{a_{i-1}+a_{i+1}}{a_i}\Big\}&\geq\frac{\sin\alpha_{i_0-1}+\sin\alpha_{i_0+1}}{\sin\alpha_{i_0}}\\
&=\cos d_{i_0}+\cos d_{i_0+1}+\frac{\cos \alpha_{i_0}}{\sin\alpha_{i_0}}\big(\sin d_{i_0+1}-\sin d_{i_0}\big)\\
&\geq2\cos\Big(\frac{\pi}{n+1}\Big).
\end{align*}

\item[(b)] Suppose that $i_0=m$. Note that we have $\alpha_{i_0}=\pi/2$ and depending on whether $n$ is even or odd, we would have $\alpha_{i_0+1}=\alpha_{i_0}$ or $\alpha_{i_0+1}=\alpha_{i_0-1}$. In either case we have $d_{i_0+1}\leq\frac{\pi}{n+1}$, and we proceed as in case (a).

\item[(c)] Suppose now that $i_0<m$ and $\alpha_{i_0+1}-\alpha_{i_0}=d_{i_0+1}\geq\frac{\pi}{n+1}$. In this case, we have $\alpha_{i_0+1}\leq\pi/2$ and
$$
\alpha_{i_0+1}\geq \alpha_{i_0}+\frac{\pi}{n+1},\quad \alpha_{i_0-1}\geq \alpha_{i_0}-\frac{\pi}{n+1}.
$$
Since $\sin x$ is an increasing function for $x\in[0,\pi/2]$ we have that
\begin{align*}
\max_{1\leq i\leq n}\Big\{\frac{a_{i-1}+a_{i+1}}{a_i}\Big\}&\geq\frac{\sin \alpha_{i_0-1}+\sin\alpha_{i_0+1}}{\sin\alpha_{i_0}}\\
&\geq\frac{\sin(\alpha_{i_0}-\frac{\pi}{n+1})+\sin(\alpha_{i_0}+\frac{\pi}{n+1})}{\sin\alpha_{i_0}}\\
&=\frac{\sin\alpha_{i_0}\cos(\frac{\pi}{n+1})-\cos\alpha_{i_0}\sin(\frac{\pi}{n+1})}{\sin\alpha_{i_0}}\\
&\qquad+\frac{\sin\alpha_{i_0}\cos(\frac{\pi}{n+1})+\cos\alpha_{i_0}\sin(\frac{\pi}{n+1})}{\sin\alpha_{i_0}}\\&=2\cos\Big(\frac{\pi}{n+1}\Big).
\end{align*}
\end{enumerate}
\end{proof}

\begin{remark}
A different proof of Lemma~\ref{l:lowerbound} can be given using the Chebyshev polynomials of the second kind \cite{MH}, which are defined by means of the recurrence relation:
$$
\left\{
\begin{array}{l}
U_0(x)=1
  \\ U_1(x)=2x\\
  U_{n}(x)=2xU_{n-1}(x)-U_{n-2}(x),\quad k=2,3,\dots\end{array}
\right.
$$
and observing that the roots of the equation $U_n(x)=0$, in increasing order, are precisely
$$
x_j^n=\cos\left(\frac{\pi(n-j+1)}{n+1}\right),\quad\text{for}\quad j=1,\dots,n.
$$
\end{remark}

\begin{thm}\label{thm:C^0_Ln}
For $n\geq2$ we have that $C^0_{L_n}=1+2\cos\Big(\frac{\pi}{n+1}\Big).$
\end{thm}

\begin{proof}
By Lemma \ref{l:regularpolygon}, we know that 
$$
C^0_{L_n}=\inf_\mu C^0_{\mu}\leq 1+2\cos\Big(\frac{\pi}{n+1}\Big).
$$

For the converse inequality, since $C^0_{L_n}<3$, using Lemma \ref{l:localmin}  we can restrict our attention to measures $\mu$ with no local minimum. Moreover, by Lemma \ref{l:symmetric}, we can also assume that the measure is symmetric with respect to $\lceil\frac{n}{2}\rceil$, and by Lemma \ref{l:C<3} that $a_i<a_{i+1}$. In this case, Lemma \ref{l:lowerbound} yields the conclusion.
\end{proof}

For a finite graph $G$, $C_{G}^0$ is related to the spectral properties of $G$, as the following theorem shows (see also \cite{BH} for further results on this topic). Recall that the adjacency matrix of a graph $G$ is the matrix $A_G$ whose $(i,j)$ entry is $1$, if the edge $(i,j)\in E_G$, and $0$, otherwise.

\begin{thm}\cite[Theorem~5]{DST}\label{thm:doubling spectra}
For every finite graph $G$, it holds that
$$
C_G^0=1+\lambda_1(A_G),
$$
where $\lambda_1(A_G)$ denotes the largest eigenvalue of the adjacency matrix of $G$. Moreover, the unique minimizer for $C_G^0$, up to multiplicative constants,  is the measure $\mu$ given by the Perron eigenvector corresponding to $\lambda_1(A_G)$.
\end{thm}

\section{$\mathbb Z$ as an infinite linear graph}\label{constz}

One can consider an asymptotic version of Theorem \ref{thm:C^0_Ln} as follows: Since 
$$
1+2\cos\Big(\frac{\pi}{n+1}\Big)\leq C_{L_n}\leq 3,
$$
taking $n\rightarrow \infty$ we have that 
$$
\lim_{n\rightarrow \infty} C_{L_n}=3.
$$
As, $\mathbb Z$ can be viewed as the limit of $L_n$, it is natural to expect that $C_{\mathbb Z}=3$. Actually, the computation of $C_{\mathbb Z}$ will be considerably simpler than (and independent of) that of $C_{L_n}$.
\medskip

For convenience, we need to recall the following well-known fact first.

\begin{lem}\label{infinitemeasure}
If $(X,d)$ is a metric space with infinite diameter, then every doubling measure $\mu$ on $X$ satisfies $\mu(X)=\infty$.
\end{lem}

\begin{proof}
Assume on the contrary that $\mu(X)<\infty$. Fix any $x\in X$. Since $\mu(X)=\lim_{r\rightarrow\infty}\mu(B(x,r))$, for every $\varepsilon >0$, we can find $r>0$ such that
$$
\mu(B(x,r))\geq (1-\varepsilon)\mu(X).
$$
Take $y\in X$ such that $d(x,y)>3r$. In particular, since $B(y,d(x,y)-r)\cap B(x,r)=\emptyset$, it follows that 
$$
\mu(B(y,d(x,y)-r))\leq \varepsilon\mu(X).
$$
On the other hand, because of our choice of $y$, we also have that
$$
B(x,r)\subset B(y,2(d(x,y)-r)).
$$
Hence,
$$
(1-\varepsilon)\mu(X)\leq \mu(B(x,r))\leq C_\mu \mu(B(y,d(x,y)-r))\leq C_\mu \varepsilon\mu(X).
$$
Taking $\varepsilon\rightarrow0$ we reach a contradiction.
\end{proof}

\begin{thm}\label{const3}
$C_{\mathbb Z}=C^0_{\mathbb Z}=3$.
\end{thm}

\begin{proof}
Let $\lambda$ denote the counting measure on $\mathbb Z$. We have that 
$$
C_{\lambda}=\sup\Big\{\frac{\lambda(B(n,2k+1))}{\lambda(B(n,k))}:n\in\mathbb Z,\,k\in\mathbb N\cup\{0\}\Big\}=\sup\Big\{\frac{4k+3}{2k+1}:k\in\mathbb N\cup\{0\}\Big\}=3.
$$
Hence, $C_{\mathbb Z}\leq 3$.  Let us now see that $C^0_{\mathbb Z}\ge3$. Let $\mu$ be any doubling measure on $\mathbb Z$. For $n\in\mathbb Z$, let $a_n=\mu(n)$. Note that we can assume that for every $n\in\mathbb Z$ 
\begin{equation}\label{localmin}
a_n>\min\{ a_{n-1},a_{n+1}\}.
\end{equation}
Indeed, otherwise for some $n\in\mathbb N$ we would have
$$
C_\mu^0\geq \frac{\mu(B(n,1))}{\mu(B(n,0))}=\frac{a_{n-1}+a_n+a_{n+1}}{a_n}\geq3.
$$

Now, condition \eqref{localmin} for every $n\in\mathbb Z$ implies that the sequence $(a_n)_{n\in\mathbb Z}$ is either monotone increasing, or monotone decreasing, or there is $n_0$ such that $(a_n)_{n\leq n_0}$ is increasing and $(a_n)_{n\geq n_0}$ is decreasing. Thus, we can consider the limits
$$
L_+=\lim_{n\rightarrow\infty} a_n\quad\text{and}\quad L_-=\lim_{n\rightarrow-\infty} a_n.
$$
Note that if $L_+$ or $L_-$ belong to $(0,\infty)$, then $C_\mu^0\geq 3$. Indeed, assume for instance $L_+\in (0,\infty)$, then for every $\varepsilon>0$ there is $N\in\mathbb N$ such that $|a_n-L_+|\leq \varepsilon$, for $n\geq N$. In particular,
$$
C_\mu^0\geq\frac{\mu(B(N+1,1))}{\mu(B(N+1,0))}=\frac{a_{N}+a_{N+1}+a_{N+2}}{a_{N+1}}\geq \frac{3(L_+-\varepsilon)}{L_++\varepsilon}\underset{\varepsilon\rightarrow0}\longrightarrow 3.
$$
Therefore, we can restrict our analysis to the following three cases:

\begin{enumerate}[leftmargin=*]
  \item[(i)] $(a_n)_{n\in\mathbb Z}$ is monotone increasing with $L_-=0$ and $L_+=\infty$, or
  \item[(ii)] $(a_n)_{n\in\mathbb Z}$ is monotone decreasing with $L_-=\infty$ and $L_+=0$, or
  \item[(iii)] there is $n_0$ such that $(a_n)_{n\leq n_0}$ is increasing, $(a_n)_{n\geq n_0}$ is decreasing, and $L_+=L_-=0$.
\end{enumerate}

Assume case (i) holds. Note that for every $n\in\mathbb N$, we have $0<\frac{a_{n-1}}{a_n}\leq1$. Suppose first that 
\begin{equation}\label{limsup}
\limsup_{n\rightarrow-\infty}\frac{a_{n-1}}{a_n}=1.
\end{equation} 
In this case, for every $\varepsilon>0$, there is $m\in\mathbb Z$ such that
$$
1-\varepsilon\leq \frac{a_{m-1}}{a_m}\leq1.
$$
Hence, we have
$$
C_\mu^0\geq\frac{\mu(B(m,1))}{\mu(B(m,0))}=\frac{a_{m-1}+a_{m}+a_{m+1}}{a_{m}}\geq 3-\varepsilon\underset{\varepsilon\rightarrow0}\longrightarrow 3.
$$

On the contrary, suppose now that \eqref{limsup} does not hold. In that case, by the classical ratio test it follows that the series $\sum_{n\leq0} a_n$ converges. Let $s=\sum_{n\leq0} a_n$. Since, by hypothesis, $L_+=\infty$, for every $M>0$,  there is $N\in\mathbb N$ such that $a_N\geq Ms$. Therefore, we have
$$
C_\mu\geq \frac{\mu(B(-N,2N))}{\mu(B(-N,N))}=\frac{\sum_{n=-3N}^{N}a_n}{\sum_{n=-2N}^{0}a_n}\geq \frac{a_N}{s}\geq M.
$$
This shows that $C_\mu=\infty$, which is a contradiction.

Case (ii), follows by symmetry. Finally, in case (iii), Lemma \ref{infinitemeasure} and the ratio test imply that either
$$
\limsup_{n\rightarrow\infty}\frac{a_{n+1}}{a_n}=1,
$$
or
$$
\limsup_{n\rightarrow-\infty}\frac{a_{n-1}}{a_n}=1,
$$
and we proceed as above to get $C_\mu^0\geq3$.
\end{proof}

A similar analysis can be made for $\mathbb N$:

\begin{thm}\label{const3:N}
$C_{\mathbb N}=C^0_{\mathbb N}=3$.
\end{thm}

\begin{proof}
Let $\lambda$ denote the counting measure on $\mathbb N$. Note that for $j\in\mathbb N$, $k\geq0$ we have that 
$$
\lambda(B(j,k))=
\left\{
\begin{array}{ccc}
 2k+1, &   & \text{ if }k\leq j-1,  \\
 j+k ,&   &   \text{ if }k> j-1.
\end{array}
\right.
$$
Hence, we get that
\begin{align*}
C_{\lambda}&=\sup\Big\{\frac{\lambda(B(j,2k+1))}{\lambda(B(j,k))}:j\in\mathbb N,\,k\in\mathbb N\cup\{0\}\Big\}\\
&=\max\bigg\{\sup\Big\{\frac{4k+3}{2k+1}:0\leq k\leq \frac{j-2}{2}\Big\}, \sup\Big\{\frac{j+2k+1}{2k+1}:\frac{j-2}{2}<k\leq j-1\Big\},\\
&\quad \quad \quad \quad \sup\Big\{\frac{j+2k+1}{j+k}:j-1<k\Big\}\bigg\}=3.
\end{align*}
Thus, $C_{\mathbb N}\leq 3$.  Let us now see that $C^0_{\mathbb N}\ge3$. Let $\mu$ be any doubling measure on $\N$. For $n\in\mathbb N$, let $\mu_n$ be the restriction of $\mu$ to $L_n$. Then, by Theorem~\ref{thm:C^0_Ln}, 
$$
1+2\cos\bigg(\frac{\pi}{n+1}\bigg)=C^0_{L_n} \le C^0_{\mu_n} \le C^0_\mu.
$$
The last inequality follows trivially since the only quotient for $\mu_n$ to calculate  $C^0_{\mu_n} $, different for the measure $\mu$ in $\N$, is the one given at the leave $n$, on which we have
$$
\frac{\mu_n(n-1)+\mu_n(n)}{\mu_n(n)}<\frac{\mu(n-1)+\mu(n)+\mu(n+1)}{\mu(n)}.
$$
Hence, 
$$
\lim_{n\to\infty}C^0_{L_n} =3\le C^0_\mu\quad \Longrightarrow\quad 3\le C^0_\N.
$$
\end{proof}

In the remaining of this section we will see that the counting measure is the only doubling minimizer on $\mathbb Z$ (up to multiplicative constant). To this end we need first a couple of lemmas:

\begin{lem}\label{l:valle}
Let $\mu$ be a doubling measure on $\mathbb Z$ and set $a_j=\mu(j)$, for $j\in\mathbb Z$. If there exist $j_1<j_2<j_3$ in $\mathbb Z$ such that $a_{j_2}<\min\{a_{j_1},a_{j_3}\}$, then $C_\mu^0>3$.
\end{lem}

\begin{proof}
Let $j_0\in[j_1,j_3]$ be such that $a_{j_0}=\min_{j\in[j_1,j_3]} a_j$. Let 
$$
j_0^+=\min\{j\geq j_0:a_j> a_{j_0}\} \quad\quad\text{and}\quad\quad j_0^-=\max\{j\leq j_0:a_j> a_{j_0}\}.
$$
Note that $j_0<j_0^+\leq j_3$ and $j_1\leq j_0^-<j_0$. Now it is easy to check that
$$
C_\mu^0\geq\max\Big\{\frac{\mu(B(j_0^+-1,1))}{\mu(B(j_0^+-1,0))},\frac{\mu(B(j_0^-+1,1))}{\mu(B(j_0^-+1,0))}\Big\}>3.
$$
\end{proof}

\begin{lem}\label{l:strictlymonotone}
Let $\mu$ be a doubling measure on $\mathbb Z$ with $C_\mu^0=3$. If there is $j_0\in\mathbb Z$ such that $a_{j_0}<a_{j_0+1}$, then $a_j<a_{j+1}$ for every $j\leq j_0$.
\end{lem}

\begin{proof}
Assuming the hypothesis we will see that $a_{j_0-1}<a_{j_0}$ and the result would follow by induction. Suppose otherwise that 
\begin{equation}\label{eq:aj0-1}
a_{j_0-1}\geq a_{j_0}.
\end{equation}
Since $C_\mu=3$, in particular we have
$$
\frac{a_{j_0-1}+a_{j_0}+a_{j_0+1}}{a_{j_0}}=\frac{\mu(B(j_0,1))}{\mu(B(j_0,0))}\leq C_\mu^0=3,
$$
which can be rewritten as
\begin{equation}\label{eq:prom}
\frac{a_{j_0-1}+a_{j_0+1}}{2}\leq a_{j_0}.
\end{equation}
Using the hypothesis, together with \eqref{eq:aj0-1} and \eqref{eq:prom}, we would have
$$
a_{j_0}<\frac{a_{j_0-1}+a_{j_0+1}}{2}\leq a_{j_0},
$$
which is a contradiction.
\end{proof}

\begin{thm}\label{t:minimizersZ}
A measure $\mu$ on $\mathbb Z$ is a doubling minimizer if and only if it is constant.
\end{thm}

\begin{proof}
Let us denote $a_j=\mu(j)$ for $j\in\mathbb Z$ and suppose $\mu$ is not constant. In that case, there exists $j_0\in\mathbb Z$ such that $a_{j_0}\neq a_{j_0+1}$. Let us assume that $a_{j_0}<a_{j_0+1}$, as the case with the opposite inequality would follow by symmetry. By Lemmas \ref{l:valle} and \ref{l:strictlymonotone}, we have to distinguish to possibilities:
\begin{enumerate}[leftmargin=*]
\item[(i)] either $a_j\leq a_{j+1}$ for every $j\in\mathbb Z$; or
\item[(ii)] there exist $j_0\leq j_1< j_2$ such that $a_j<a_{j+1}$ for $j\leq j_1$, $a_j=a_{j_2}$ for $j_1<j\leq j_2$ and $a_j>a_{j+1}$ for $j\geq j_2$.
\end{enumerate}

First, let us suppose case (i) holds. For every $n\in\mathbb Z$ let $\mu_n$ be the measure given, for $j\in\mathbb Z$, by
$$
\mu_n(j)=\mu(n+j)+\mu(-j).
$$
As a direct consequence of Lemma \ref{l:holder}, it follows that $C_{\mu_n}=3$ for every $n\in\mathbb Z$. Note that if $\lim_{j\rightarrow\infty} a_j=\infty$, then $\lim_{|j|\rightarrow\infty}\mu_0(j)=\infty$; thus, if we set $i_0$ such that $\mu_0(i_0)=\min_{j\in\mathbb Z}\mu_0(j)$, then for $N>|i_0|$ large enough we have in particular that $\mu_0(N)=\mu_0(-N)>\mu_0(i_0)$, and by Lemma \ref{l:valle} we would have that $C_{\mu_0}^0>3$, which is a contradiction. Hence, we can assume that $\lim_{j\rightarrow\infty} a_j=\sup_{j\in\mathbb Z} a_j=M\in\mathbb R_+$. Let also $m=\lim_{j\rightarrow-\infty}a_j=\inf_{j\in\mathbb Z}a_j$, which satisfies $m\geq0$.

We claim that in this case, for some $n\in\mathbb Z$ we must have that the measure $\mu_n$ is not a multiple of the counting measure. Indeed, if this were not the case, then for every $n\in\mathbb Z$ there would be $t_n$ such that for every $j\in \mathbb Z$
$$
a_{n+j}+a_{-j}=t_n.
$$
In particular, taking $j\rightarrow \infty$, we get that for every $n\in\mathbb Z$
$$
t_n=\lim_{j\rightarrow\infty} (a_{n+j}+a_{-j})=M+m.
$$
Therefore, for every $n,j\in\mathbb Z$ we have that $a_{n+j}-a_{-j}=M+m$, which implies in particular that 
$$a_n=M+m-a_0$$
for every $n\in\mathbb Z$. This is a contradiction with the assumption that $\mu$ was not constant.

Therefore, we can take some $n\in\mathbb Z$ such that $\mu_n$ is not constant. In that case, Lemma \ref{l:valle} implies that $\mu_n$ must have the form as in (ii) above and we would proceed as below.

Suppose now that we have a non-constant measure $\mu$ with $C_\mu=3$ satisfying condition (ii). For simplicity, and without loss of generality using a translation we can assume that $a_0=\sup_{j\in\mathbb Z} a_j$. As before let $M=a_0$ and $m=\lim_{|j|\rightarrow \infty} a_j$. Let $0<\varepsilon<\frac{M-m}{3}$. Let $N\in\mathbb N$ large enough so that for every $|j|\geq N$ we have $|a_j-m|<\varepsilon$. Let us consider a new measure $\tilde \mu$ obtained by combining the original one and a translation: For $j\in\mathbb Z$ let
$$
\tilde\mu(j)=a_j+a_{2N+j}.
$$
It follows by Lemma \ref{l:holder} that $C_{\tilde\mu}=3$. However, we have that
\begin{align*}
\tilde\mu(0)&=a_0+a_{2N}=M+a_{2N}>M+m-\varepsilon,\\
\tilde\mu(-2N)&=a_{-2N}+a_{0}=M+a_{-2N}>M+m-\varepsilon,\\
\tilde\mu(-N)&=a_{-N}+a_N<2m+2\varepsilon.
\end{align*}
By our choice of $\varepsilon$ we have that $\tilde\mu(-N)<\min\{\tilde\mu(-2N),\tilde\mu(0)\}$, which is a contradiction with Lemma \ref{l:valle} and the fact that $C_{\tilde\mu}=3$.
\end{proof}

\begin{remark}
It is noteworthy to mention that the symmetries of $\mathbb Z$ (translations and reflections) are implicitly playing a role in the above proof. We also refer the interested reader to \cite{DST} for further examples of graphs for which doubling constants can be computed, and in particular several instances where the automorphism group of the graph can be exploited.
\end{remark}

It is somehow curious that $DM(\mathbb N)$ does not reduce to the ray generated by the counting measure $\lambda$:

\begin{example}\label{examplenonunique}
For $\alpha\in[\frac12,1]$, if $\lambda_\alpha$ denotes the measure on $\mathbb N$ given by
$$
\lambda_\alpha(j)=
\left\{
\begin{array}{ccc}
\alpha, &   & \text{ if }j=1,  \\
1, &   &  \text{ if }j> 1,
\end{array}
\right.
$$
then $C_{\lambda_\alpha}=3$.
Indeed, by Theorem~\ref{const3:N}, $3=C_\N\le C_{\lambda_\alpha}$ and, for $j\in\mathbb N$, if $k<j-1$ we have that
$$
\frac{\lambda_\alpha(B(j,2k+1))}{\lambda_\alpha(B(j,k))}\leq \frac{\lambda(B(j,2k+1))}{\lambda(B(j,k))}\leq 3;
$$
whereas for $k\geq j-1$ we have
$$
\frac{\lambda_\alpha(B(j,2k+1))}{\lambda_\alpha(B(j,k))}= \frac{\alpha+j+2k}{\alpha+j+k-1}\leq 2+\frac{1}{2k+1}\leq 3.
$$
\end{example}

\section{Estimates for $C_{L_n}$}\label{secc5}

The purpose of this section is to show that the measures on $L_n$ we are interested for computing the infimum $C_{L_n}$ have some nice properties. In particular, it is not necessary to compute all possible quotients of measures of balls appearing in the definition of $C_\mu$. This will be useful later on to show that, actually, $C_{L_n}$ is monotone increasing in $n$.

For a measure $\mu$ on $L_n$, $n\ge 3$, let 
$$
M_1(\mu)=\sup\Big\{\frac{\mu(B(1,2k+1))}{\mu(B(1,k))}:0\le k<\Big\lceil \frac{n-2}{3}\Big\rceil\Big\}
$$
and
\begin{align*}
 M_2(\mu)=&\sup \Big\{ \frac{\mu(B(j,2k+1))}{\mu(B(j,k))}: 1<j< \Big\lceil \frac{n}{2}\Big\rceil, \\[.3cm]
&\qquad      \qquad
  \frac{ \lceil \frac{n}{2}\rceil-j-1}{2}<k<\min\Big\{\frac{j-2}{2}, \Big\lceil\frac{n-2j}{3}\Big\rceil,\Big\lceil \frac{n}{2}\Big\rceil-j\Big\} \Big\}.
\end{align*}
Notice that the balls in the denominator of $M_1(\mu)$ do not contain the vertex $\lceil \frac{n}{2}\rceil$. On the other hand, balls in the denominator of $M_2(\mu)$ do not contain the vertex $\lceil \frac{n}{2}\rceil$ neither, whereas balls in the numerator of $M_2(\mu)$ always contain the vertex $\lceil \frac{n}{2}\rceil$.

\medskip

The following result is the main estimate we will use to study $C_{L_n}$:

\begin{thm}\label{t:M1M2M3}
Let $\mu$ be a symmetric measure on $L_n$. If $C_\mu^0<3$, then
$$
C_\mu= \max\{M_1(\mu),M_2(\mu),C_\mu^0\}.
$$
\end{thm}

\begin{proof}
The proof will be split in a series of claims. Given $\mu$ with  $C_\mu^0<3$, by Lemma \ref{l:C<3}, we have that $a_j=\mu(j)$ satisfies that $a_i<a_j<\frac{j}{i}a_i$ for $1\leq i<j \leq\lceil\frac{n}{2}\rceil$. Also, since $a_j=a_{n+1-j}$, when computing $C_\mu$ we only need to consider balls centered at a vertex $1\leq j\leq \lceil\frac{n}{2}\rceil$. In what follows, we will use the convention that $a_l=0$, for $l\leq0$ or $l>n$.

First, note that to calculate $C_\mu$ we only need quotients of balls for which the ball in the denominator is contained in $\{1,\ldots,\lceil\frac{n}{2}\rceil\}$:

\begin{fact}\label{j+k leq n/2}
$$C_\mu=\sup\Big\{\frac{\mu(B(j,2k+1))}{\mu(B(j,k))}:1\leq j\leq \Big\lceil\frac{n}{2}\Big\rceil,j+k\leq\Big\lceil\frac{n}{2}\Big\rceil\Big\}.$$
\end{fact}

\begin{proof}
Let $1\leq j\leq \lceil\frac{n}{2}\rceil$ and $k>\lceil\frac{n}{2}\rceil-j$. Note that
$$
\mu(B(j,k))=\sum_{i=\lceil\frac{n}{2}\rceil-j+1}^{\min\{k,j-1\}} a_{j-i} + \mu\bigg(B\bigg(j,\Big\lceil\frac{n}{2}\Big\rceil-j\bigg)\bigg)+\sum_{i=\lceil\frac{n}{2}\rceil-j+1}^{\min\{k,n-j\}}a_{j+i},
$$
and
\begin{align*}
\mu(B(j,2k+1))&=\sum_{i=\lceil\frac{n}{2}\rceil-j+1}^{\min\{k,j-1\}} (a_{j-2i-1}+ a_{j-2i} )+ \mu\bigg(B\bigg(j,2\bigg(\Big\lceil\frac{n}{2}\Big\rceil-j\bigg)+1\bigg)\bigg)\\
&\qquad\qquad +\sum_{i=\lceil\frac{n}{2}\rceil-j+1}^{\min\{k,n-j\}}(a_{j+2i}+a_{j+2i+1}).
\end{align*}

For each $\lceil\frac{n}{2}\rceil-j<i\leq k$, we have $a_{j+2i}+a_{j+2i+1}\leq 2a_{j+i}$ and if  $j-i\geq1$ we also have $a_{j-2i-1}+a_{j-2i}\leq 2 a_{j-i}$. Therefore, by Lemma~\ref{l:holder}, we have that
\begin{align*}
\frac{\mu(B(j,2k+1))}{\mu(B(j,k))}&\leq \max\Big\{\frac{\mu(B(j,2(\lceil\frac{n}{2}\rceil-j)+1))}{\mu(B(j,\lceil\frac{n}{2}\rceil-j))},\frac{a_{j-2i-1}+a_{j-2i}}{a_{j-i}},\frac{a_{j+2i}+a_{j+2i+1}}{a_{j+i}}\Big\}\\
&\leq \max\Big\{\frac{\mu(B(j,2(\lceil\frac{n}{2}\rceil-j)+1))}{\mu(B(j,\lceil\frac{n}{2}\rceil-j))},2\Big\}.
\end{align*}
Since $C_{\mu}\ge2$ \cite{ST}, the conclusion follows.
\end{proof}

Except for $j= \lceil\frac{n}{2}\rceil$, we can actually assume that the ball in the denominator does not contain the vertex $\lceil\frac{n}{2}\rceil$:

\begin{fact}\label{B(j,k)}
For $1\leq j<\lceil\frac{n}{2}\rceil$, and $k=\lceil\frac{n}{2}\rceil-j$ we have
$$
\frac{\mu(B(j,2k+1))}{\mu(B(j,k))}\leq \max\Big\{\frac{\mu(B(j,2k-1))}{\mu(B(j,k-1))},2\Big\}.
$$
\end{fact}

\begin{proof}
As above, we have
\begin{align*}
\frac{\mu(B(j,2k+1))}{\mu(B(j,k))}&\leq \frac{\mu(B(j,2k-1))+a_{j+2k}+a_{j+2k+1}}{\mu(B(j,k-1))+a_{j+k}}\\
&\leq \max\Big\{\frac{\mu(B(j,2k-1))}{\mu(B(j,k-1))},\frac{a_{j+2k}+a_{j+2k+1}}{a_{j+k}}\Big\}.
\end{align*}
The conclusion follows from the fact that $a_{j+k}=a_{\lceil\frac{n}{2}\rceil}\geq \max\{a_{j+2k},a_{j+2k+1}\}$.
\end{proof}

\begin{fact}\label{k>n+1-2j/3}
If $k\geq\lceil\frac{n-2j}{3}\rceil$, then 
$$
\frac{\mu(B(j,2k+1))}{\mu(B(j,k))}\leq\max\Big\{\frac{\mu(B(j,2k-1))}{\mu(B(j,k-1))},\frac{\mu(B(j+k,1))}{\mu(B(j+k,0))},\frac{\mu(B(j-k,1))}{\mu(B(j-k,0))}\Big\},
$$
the last term in the max appearing only when $k<j$.
\end{fact}

\begin{proof}
Suppose $k<j$, then
$$
\frac{\mu(B(j,2k+1))}{\mu(B(j,k))}\leq\max\Big\{\frac{a_{j-2k-1}+a_{j-2k}}{a_{j-k}},\frac{\mu(B(j,2k-1))}{\mu(B(j,k-1))},\frac{a_{j+2k}+a_{j+2k+1}}{a_{j+k}}\Big\}.
$$
Since $k\geq\lceil\frac{n-2j}{3}\rceil$, we have that $j+2k>n+1-(j+k+1)$ and $j+2k+1\geq n+1-(j+k)$. Hence, by monotonicity and symmetry of the weights we have
$$
\frac{a_{j+2k}+a_{j+2k+1}}{a_{j+k}}\leq\frac{a_{n+1-(j+k+1)}+a_{n+1-(j+k)}}{a_{j+k}}=\frac{a_{j+k}+a_{j+k+1}}{a_{j+k}}<\frac{\mu(B(j+k,1))}{\mu(B(j+k,0))}.
$$
Using the monotonicity we also have
$$
\frac{a_{j-2k-1}+a_{j-2k}}{a_{j-k}}\leq\frac{a_{j-k-1}+a_{j-k}}{a_{j-k}}<\frac{\mu(B(j-k,1))}{\mu(B(j-k,0))}.
$$
The estimate when $j\leq k$ is similar, only the term $(a_{j-2k-1}+a_{j-2k})/a_{j-k}$ does not appear in the first quotient.
\end{proof}

\begin{fact}\label{alfa}
Let $\alpha\in(0,1)$. If $k\leq\frac{\alpha j-1}{2-\alpha}$, $j+2k+1\leq\lceil \frac{n}{2}\rceil$, then 
$$
\frac{\mu(B(j,2k+1))}{\mu(B(j,k))}\leq\max\Big\{\frac{\mu(B(j,2k-1))}{\mu(B(j,k-1))},2+\alpha\Big\}.
$$
\end{fact} 

\begin{proof}
For $\alpha\in(0,1)$ we have $k\le \frac{\alpha j-1}{2-\alpha}<j$, so we proceed as above:
$$
\frac{\mu(B(j,2k+1))}{\mu(B(j,k))}\leq\max\Big\{\frac{a_{j-2k-1}+a_{j-2k}}{a_{j-k}},\frac{\mu(B(j,2k-1))}{\mu(B(j,k-1))},\frac{a_{j+2k}+a_{j+2k+1}}{a_{j+k}}\Big\}.
$$
By monotonicity $\frac{a_{j-2k-1}+a_{j-2k}}{a_{j-k}}<2$. Now, as  $j+2k+1\leq\lceil \frac{n}{2}\rceil$, by Lemma \ref{l:C<3} we have 
$$
\frac{a_{j+2k}+a_{j+2k+1}}{a_{j+k}}<\frac{j+2k+j+2k+1}{j+k}\leq 2+\alpha.
$$
\end{proof}

\begin{fact}\label{k geq j-1}
If $k\geq j-1$ then 
$$
\frac{\mu(B(j,2k+1))}{\mu(B(j,k))}\leq \frac{\mu(B(1,2(j+k-1)+1))}{\mu(B(1,j+k-1))}.
$$
\end{fact}

\begin{proof}
If $k\geq j-1$, then $B(j,k)=B(1,j+k-1)$. Hence,
$$
\frac{\mu(B(j,2k+1))}{\mu(B(j,k))}= \frac{\mu(B(1,j+2k))}{\mu(B(1,j+k-1))}\leq \frac{\mu(B(1,2(j+k-1)+1))}{\mu(B(1,j+k-1))}.
$$
\end{proof}

\begin{fact}\label{k geq j-1/2}
If $\frac{j-2}2\leq k<j-1$ then 
$$
\frac{\mu(B(j,2k+1))}{\mu(B(j,k))}\leq \frac{\mu(B(1,4k+1))}{\mu(B(1,2k))}.
$$
\end{fact}

\begin{proof}
As $k<j-1$, we have that $B(j,k)$ consists of $2k+1$ vertices. By monotonicity of the measure we have then that
$$
\mu(B(j,k))\geq \mu(B(1,2k)).
$$
Now, since $k\geq\frac{j-2}{2}$ we have 
$$
B(j,2k+1)\subset B(1,j+2k)\subset B(1,4k+1).
$$
Therefore, we get
$$
\frac{\mu(B(j,2k+1))}{\mu(B(j,k))}\leq \frac{\mu(B(1,4k+1))}{\mu(B(1,2k))}.
$$
\end{proof}

Now, let us finish the proof of the theorem. First note that using Claim \ref{k>n+1-2j/3} for $j=1$ it follows that
$$
M_1(\mu)\le\max\bigg\{\sup\Big\{\frac{\mu(B(1,2k+1))}{\mu(B(1,k))}: k\geq0\Big\},C^0_\mu\bigg\}.
$$

Given $1< j\leq \lceil \frac{n}{2}\rceil$ and $1\leq k\leq  \lceil \frac{n}{2}\rceil$, in order to bound the quotient
$$
\frac{\mu(B(j,2k+1))}{\mu(B(j,k))}
$$ 
by Claims \ref{j+k leq n/2} and \ref{B(j,k)}, we can assume without loss of generality that $j+k<  \lceil \frac{n}{2}\rceil$. 

Suppose first that $k\geq\frac{j-2}{2}$, then by Claims \ref{k geq j-1} and \ref{k geq j-1/2}, 
$$
\frac{\mu(B(j,2k+1))}{\mu(B(j,k))}\leq M_1(\mu).
$$ 
Hence, we can assume $k<\frac{j-2}{2}$. Let $\alpha=\frac{2j-2}{3j-2}$. By Claim \ref{alfa}, for $k\leq\frac{\alpha j-1}{2-\alpha}=\frac{j-2}{2}$, as long as $j+2k+1\leq\lceil \frac{n}{2}\rceil$, we get
$$
\frac{\mu(B(j,2k+1))}{\mu(B(j,k))}\leq \max\Big\{\frac{\mu(B(j,2k-1))}{\mu(B(j,k-1))},2+\frac{2j-2}{3j-1}\Big\}.
$$ 
Since $j\geq2$, for $n\geq5$ we have that 
$$
2+\frac{2j-2}{3j-2}\leq2+\frac23\leq1+2\cos\Big(\frac{\pi}{n+1}\Big)=C_{L_n}^0\leq C_\mu^0.
$$
Hence, for $n\geq 5$, as long as $j+2k+1\leq \lceil \frac{n}{2}\rceil$, we can inductively reduce the radius $k$, and we get
$$
\frac{\mu(B(j,2k+1))}{\mu(B(j,k))}\leq C_\mu^0.
$$ 
Note that for $n\leq 4$ one can directly see that the result holds.

An application of Claim \ref{k>n+1-2j/3} yields that the only quotients of balls which are not bounded by $M_1(\mu)$ or $C_\mu^0$, are those appearing in the expression of $M_2(\mu)$. This finishes the proof.
\end{proof}

As an application of Theorem \ref{t:M1M2M3}, we can show the following monotonicity property of $C_{L_n}$:

\begin{prop}\label{p:C_Ln non-decreasing}
For every $n\in \mathbb N$, $C_{L_n}\leq C_{L_{n+1}}$.
\end{prop}

\begin{proof}
By Lemma \ref{l:symmetric}, it is enough to show that for every symmetric measure $\mu$ on $L_{n+1}$ there is a measure $\nu$ on $L_n$ such that $C_\nu\leq C_\mu$. We know $C_{L_n}\leq 3$. If  $C_{L_n}=3$, the statement holds trivially (in fact, we will see in Theorem~\ref{lessthan3} that this case never happens). Otherwise, if  $C_{L_n}<3$, we can always assume without loss of generality that $C_\mu<3$. Given such $\mu$ on $L_{n+1}$, let us consider the measure $\nu$ on $L_n$ defined for $1\leq j\leq n$, as follows
$$
\nu(j)=
\left\{
\begin{array}{lc}
 \mu(j) & \text{ if }j<\lceil\frac{n+1}{2}\rceil,  \\
 \mu(j+1) &   \text{ if }j\geq\lceil\frac{n+1}{2}\rceil.
\end{array}
\right.
$$

Let $0\leq k<\lceil\frac{n-2}{3}\rceil$. If follows that $k+1<\lceil\frac{n+1}{2}\rceil$ and
$$
\frac{\nu(B(1,2k+1))}{\nu(B(1,k))}=\left\{
\begin{array}{lc}
 \frac{\mu(B(1,2k+1))}{\mu(B(1,k))} & \text{ if }2k+2<\lceil\frac{n+1}{2}\rceil,  \\
\frac{\mu(B(1,2k+2))-\mu(\lceil\frac{n+1}{2}\rceil)}{\mu(B(1,k))} &   \text{ if }2k+2\geq\lceil\frac{n+1}{2}\rceil.
\end{array}
\right.
$$
Since $\mu(2k+3)\leq\mu(\lceil\frac{n+1}{2}\rceil)$, it follows that $M_1(\nu)\leq M_1(\mu)$.

Similarly, we have
$$
\frac{\nu(B(j,1))}{\nu(B(j,0))}=\left\{
\begin{array}{ll}
 \frac{\mu(B(j,1))}{\mu(B(j,0))} & \text{ if }j<\lceil\frac{n+1}{2}\rceil-1,  \\
 \frac{\mu(\lceil\frac{n+1}{2}\rceil-2)+\mu(\lceil\frac{n+1}{2}\rceil-1)+\mu(\lceil\frac{n+1}{2}\rceil+1)}{\mu(\lceil\frac{n+1}{2}\rceil-1)} & \text{ if }j=\lceil\frac{n+1}{2}\rceil-1,  \\
 \frac{\mu(\lceil\frac{n+1}{2}\rceil-1)+\mu(\lceil\frac{n+1}{2}\rceil+1)+\mu(\lceil\frac{n+1}{2}\rceil+2)}{\mu(\lceil\frac{n+1}{2}\rceil+1)}& \text{ if }j=\lceil\frac{n+1}{2}\rceil,  \\
 \frac{\mu(B(j+1,1))}{\mu(B(j+1,0))} & \text{ if }j\geq \lceil\frac{n+1}{2}\rceil+1.
\end{array}
\right.
$$
Note that 
$$
\frac{\mu(\lceil\frac{n+1}{2}\rceil-2)+\mu(\lceil\frac{n+1}{2}\rceil-1)+\mu(\lceil\frac{n+1}{2}\rceil+1)}{\mu(\lceil\frac{n+1}{2}\rceil-1)}\leq\frac{\mu(B(\lceil\frac{n+1}{2}\rceil-1,1))}{\mu(B(\lceil\frac{n+1}{2}\rceil-1,0))},
$$
and also that
$$
\frac{\mu(\lceil\frac{n+1}{2}\rceil-1)+\mu(\lceil\frac{n+1}{2}\rceil+1)+\mu(\lceil\frac{n+1}{2}\rceil+2)}{\mu(\lceil\frac{n+1}{2}\rceil+1)}\leq\frac{\mu(B(\lceil\frac{n+1}{2}\rceil+1,1))}{\mu(B(\lceil\frac{n+1}{2}\rceil+1,0))}.
$$
Thus, $C_\nu^0\leq C_\mu^0$.

Finally, for $j,k$ as in the formula for $M_2(\nu)$, we have $j+k<\lceil\frac{n}{2}\rceil\leq \lceil\frac{n+1}{2}\rceil$, which implies that $\nu(B(j,k))=\mu(B(j,k))$. Hence, we have
$$
\frac{\nu(B(j,2k+1))}{\nu(B(j,k))}=\left\{
\begin{array}{ll}
 \frac{\mu(B(j,2k+1))}{\mu(B(j,k))} & \text{ if }j+2k+1<\lceil\frac{n+1}{2}\rceil,  \\
 \frac{\mu(B(j,2k+2))-\mu(\lceil\frac{n+1}{2}\rceil)}{\mu(B(j,k))} & \text{ if }j+2k+1\geq \lceil\frac{n+1}{2}\rceil.
\end{array}
\right.
$$
Since, $\mu(j+2k+2)\leq \mu(\lceil\frac{n+1}{2}\rceil)$, it follows that $M_2(\nu)\leq C_\mu$. 
The conclusion follows by Theorem \ref{t:M1M2M3}.
\end{proof}

\section{$C_\mu$ for a doubling minimizer on $L_n$}\label{secct6}

We will show now that for doubling minimizers on $L_n$ one can actually get rid of $M_2(\mu)$ in Theorem \ref{t:M1M2M3}. We emphasize that for this we have to make use of the fact that a measure $\mu$ is a doubling minimizer, and not only that $C_\mu^0<3$. 

To this end, let us introduce the following notation: given a measure $\mu$ on $L_n$, a vertex $1\leq i\leq n$ and $\theta\in\mathbb R$, let $\mu_{i,\theta}$ be the measure given by
\begin{equation}\label{eq:perturbmu}
\mu_{i,\theta}(j)=\left\{
\begin{array}{ll}
 \mu(j) &  \text{ if }j\neq i,  \\
 \mu(i)+\theta & \text{ if }j=i.  
\end{array}
\right.
\end{equation}

\begin{prop}\label{t:minimizer}
For every $n\in \mathbb N$, if $\mu$ is a symmetric measure on $L_n$ such that $C_{L_n}=C_\mu$, then 
$$
C_\mu=\max\{M_1(\mu),C_\mu^0\}.
$$
\end{prop}

\begin{proof}
Suppose $\mu$ is a symmetric measure on $L_n$ such that $C_{L_n}=C_\mu>\max\{M_1(\mu),C_\mu^0\}$. In particular, $C_\mu\mu(\lceil\frac{n}{2}\rceil)-\mu(B(\lceil\frac{n}{2}\rceil,1))>0$. Hence, we can take $0<\varepsilon<\frac{C_\mu\mu(\lceil\frac{n}{2}\rceil)-\mu(B(\lceil\frac{n}{2}\rceil,1))}{C_\mu-1}$ and consider $\nu=\mu_{\lceil\frac{n}{2}\rceil,-\varepsilon}$ as in \eqref{eq:perturbmu} (note that $\varepsilon<\mu(\lceil\frac{n}{2}\rceil)$).

Observe that for $1\leq j<\lceil\frac{n}{2}\rceil$, we have 
$$
\frac{\nu(B(j,1))}{\nu(B(j,0))}\leq\frac{\mu(B(j,1))}{\mu(B(j,0))}\leq C_\mu^0.
$$
By our choice of $\varepsilon$ we also have
$$
\frac{\nu(B(\lceil\frac{n}{2}\rceil,1))}{\nu(B(\lceil\frac{n}{2}\rceil,0))}\leq\frac{\mu(B(\lceil\frac{n}{2}\rceil,1))-\varepsilon}{\mu(B(\lceil\frac{n}{2}\rceil,0))-\varepsilon}< C_\mu.
$$
Therefore, $C_\nu^0< C_\mu$. Now for $k<\lceil\frac{n-2}{3}\rceil< \lceil\frac{n}{2}\rceil$ we have
$$
\frac{\nu(B(1,2k+1))}{\nu(B(1,k))}\leq\frac{\mu(B(1,2k+1))}{\mu(B(1,k))}\leq M_1(\mu).
$$
Thus, $M_1(\nu)\leq M_1(\mu)<C_\mu$.

Finally, for $1<j<\lceil\frac{n}{2}\rceil$ and $\frac{\lceil\frac{n}{2}\rceil-j-1}{2}<k<\lceil\frac{n}{2}\rceil-j$, we have
$$
\frac{\nu(B(j,2k+1))}{\nu(B(j,k))}=\frac{\mu(B(j,2k+1))-\varepsilon}{\mu(B(j,k))}< C_\mu.
$$
Hence, $M_2(\nu)<C_\mu$. Since $C_{\nu}^0<3$, we can apply Theorem \ref{t:M1M2M3}, and it follows that $C_\nu= \max\{M_1(\nu), M_2(\nu),C_\nu^0\}<C_\mu=C_{L_n}$. This is a contradiction, so $C_\mu=\max\{M_1(\mu),C_\mu^0\}$ as claimed.
\end{proof}

In order to improve this result we need to estimate the constants associated to two measures playing a key role:  

\begin{lem}\label{l:senoycontar}
Let $\sigma,\lambda$  be the measures on $L_n$ given by $\sigma(j)=\sin\Big(\frac{j\pi}{n+1}\Big)$ and $\lambda(j)=1$ for every $1\leq j\leq n$. We have that
\begin{enumerate}[leftmargin=*]
\item[(i)]  $C_\sigma^0=1+2\cos(\pi/(n+1))$, $M_1(\sigma)\geq 4$ for $n$ large enough, $M_2(\sigma)\leq 2$.
\medskip

\item[(ii)] $C_\lambda^0=3$, $M_1(\lambda)=2$, $M_2(\lambda)=7/3$.
\end{enumerate}
\end{lem}

\begin{proof}	
\textit{(i)}  $C_\sigma^0$ has been computed in Lemma \ref{l:regularpolygon}. For $M_1(\sigma)$ note that given $\varepsilon>0$, there is $m\in\mathbb N$ such that $\cos(\pi/m)>(1-\varepsilon)$. Hence, for large enough $n$, using the well-known fact that $\sin  x\leq x\leq\tan x $ for $x\in(0,\pi/2)$, we deduce that for $j<\frac{n+1}{m}$ we have
$$\sin\Big(\frac{j\pi}{n+1}\Big)\leq\frac{j\pi}{n+1}\leq\frac{\sin(\frac{j\pi}{n+1})}{1-\varepsilon}.$$
Hence,
$$
M_1(\sigma)\geq\frac{\sum_{j=1}^{\frac{n+1}{m}}\sin\Big(\frac{j\pi}{n+1}\Big)}{\sum_{j=1}^{\frac{n+1}{2m}}\sin\Big(\frac{j\pi}{n+1}\Big)}\geq(1-\varepsilon)\frac{\sum_{j=1}^{\frac{n+1}{m}}j}{\sum_{j=1}^{\frac{n+1}{2m}}j}=4(1-\varepsilon)\frac{n+m+1}{n+2m+1}
$$
which approximates 4 for large enough $n$. Now, in order to estimate $M_2(\sigma)$ we use that $\sin\big(\frac{(j-i)\pi}{n+1}\big)+\sin\big(\frac{(j+i)\pi}{n+1}\big)=2\sin(\frac{j\pi}{n+1})\cos(\frac{i\pi}{n+1})$. Thus, for $j,k$ as in the expression of $M_2$, we have
\begin{align*}
\frac{\sigma(B(j,2k+1))}{\sigma(B(j,k))}&=\frac{\sum_{i=0}^{2k+1}\sin(\frac{(j-i)\pi}{n+1})+\sin(\frac{(j+i)\pi}{n+1})}{\sum_{i=0}^{k}\sin(\frac{(j-i)\pi}{n+1})+\sin(\frac{(j+i)\pi}{n+1})}\\
&=\frac{\sum_{i=0}^{2k+1}\cos(\frac{i\pi}{n+1})}{\sum_{i=0}^{k}\cos(\frac{i\pi}{n+1})}\\
&\leq\frac{\sum_{i=0}^{\min\{2k+1,\frac{n+1}{2}\}}\cos(\frac{i\pi}{n+1})}{\sum_{i=0}^{\min\{k,\frac{n-1}{4}\}}\cos(\frac{i\pi}{n+1})}\leq 2,
\end{align*}
where the last inequality follows from the monotonicity of cosine.

\textit{(ii)}  For the counting measure, it is clear that  $C_\lambda^0=3$, $M_1(\lambda)=2$ and $M_2(\lambda)=\sup_k \frac{4k+3}{2k+1}= 7/3$.
\end{proof}

\begin{lem}\label{prevlemm}
$
C_{L_n}^0=C_{L_n}\iff 2\le n\le8.
$\end{lem}

\begin{proof}
The fact that, for $2\le n\le8$ we have that $C_{L_n}^0=C_{L_n}$ is a straightforward calculation and can be computed using an easy optimization argument. Let us now see what happens if $n\ge9$. Let $\sigma$ be as in Lemma~\ref{l:senoycontar}.  We claim that if  $C_{L_n}^0=C_{L_n}$, then 
\begin{equation}\label{igualdad}
C_\sigma=C_{L_n}^0.
\end{equation}
Indeed, if  $\mu$ is a minimizer for $C_{L_n}$ (see Proposition~\ref{p:DMconvex}), then 
$$
C^0_\mu\le C_\mu=C_{L_n}=C_{L_n}^0\le C_\mu^0,
$$
and, using the uniqueness of the Perron eigenvector in Theorem~\ref{thm:doubling spectra}, we obtain that, for a positive constant $c>0$, $\mu=c\sigma$. Therefore, $C_\sigma=C_\mu=C_{L_n}=C_{L_n}^0$.

\medskip

We now estimate $M_1(\sigma)$ from below, using the ball of radius $k=1$:

\begin{align*}
M_1(\sigma)&\ge\frac{\sin\frac{\pi}{n+1}+\sin\frac{2\pi}{n+1}+\sin\frac{3\pi}{n+1}+\sin\frac{4\pi}{n+1}}{\sin\frac{\pi}{n+1}+\sin\frac{2\pi}{n+1}}\\[.3cm]
&=\frac{\sin\frac{5\pi}{2(n+1)}\sin\frac{2\pi}{n+1}}{\sin\frac{\pi}{2(n+1)}\sin\frac{\pi}{n+1}(1+2\cos\frac{\pi}{n+1})}\\[.3cm]
&=\frac{2\sin\frac{5\pi}{2(n+1)}\cos\frac{\pi}{n+1}}{\sin\frac{\pi}{2(n+1)} (1+2\cos\frac{\pi}{n+1})}.
\end{align*}
Thus, in order to prove that $M_1(\sigma)>1+2\cos(\pi/(n+1))$, it suffices  to study when the function
$$
f(x)=\frac{\sin5x\cos2x}{\sin x(1+2\cos2x)^2}-\frac12
$$
satisfies that $f(x)>0$. But, it is easy to see that 
$$
f(x)=\frac{2 \cos6 x-1 }{2 (1 + 2 \cos2 x)^2}>0\iff x<\frac{\pi}{18}.
$$
Hence, with $x=\frac{\pi}{2(n+1)}$, we conclude that 
$$
M_1(\sigma)>1+2\cos\Big(\frac{\pi}{n+1}\Big)\iff n\ge9.
$$
Therefore, if $n\ge9$ we get that $C_\sigma\ge M_1(\sigma)>C_{L_n}^0$, which, by \eqref{igualdad}, implies that $C_{L_n}^0<C_{L_n}$.
\end{proof}

\begin{prop}\label{p:minimizer}
If $\mu$ is a symmetric measure on $L_n$ such that $C_\mu=C_{L_n}$, then 
$$M_1(\mu)=C_\mu^0=C_{L_n}.$$
\end{prop}

\begin{proof}
If $C_\mu=C_{L_n}$, by Proposition \ref{t:minimizer} we know $C_\mu=\max\{C_\mu^0,M_1(\mu)\}$. In order to prove the statement we will proceed by contradiction. Suppose first $M_1(\mu)<C_{L_n}$. By continuity of the function $M_1(\cdot)$ at $\mu$, there is $\varepsilon>0$ such that for any measure $\nu$ on $L_n$ with $\max_{1\leq j\leq n}\{\nu(j)\}\leq 1$, we have that $M_1(\mu+\varepsilon\nu)< C_{L_n}$. Let $\sigma(j)=\sin(\frac{j\pi}{n+1})$. Also, we have that
$$
C_{\mu+\varepsilon\sigma}^0\leq \max\{C_\mu^0,C_\sigma^0\}\leq C_{L_n},
$$
but since $C_\sigma^0=1+2\cos(\frac{\pi}{n+1})<C_{L_n}$, for $n\ge9$ (see Lemma~\ref{prevlemm}), we must have $C_{\mu+\varepsilon\sigma}^0<C_{L_n}$. Similarly, we have
$$
M_2(\mu+\varepsilon\sigma)\leq\max\{M_2(\mu),M_2(\sigma)\},
$$
but since $M_2(\sigma)\leq 2< C_{L_n}$ (by Lemma \ref{l:senoycontar}) we must have $M_2(\mu+\varepsilon\sigma)<C_{L_n}$. Thus, by Theorem \ref{t:M1M2M3} we would have that $C_{\mu+\varepsilon\sigma}<C_{L_n}$. This is a contradiction.

Now, suppose that $C_\mu^0<C_{L_n}$. In that case, by continuity of $C_\mu^0$, there is  $\varepsilon>0$ such that for any measure $\nu$ on $L_n$ with $\max_{1\leq j\leq n}\{\nu(j)\}\leq 1$, we have that $C_{\mu+\varepsilon\nu}^0< C_{L_n}$. Let $\lambda(j)=1$ for every $1\leq j\leq n$. Then, $C_{\mu+\varepsilon\lambda}^0< C_{L_n}$. Also we have
$$
M_1(\mu+\varepsilon\lambda)\leq \max\{M_1(\mu),M_1(\lambda)\}
$$
and 
$$
M_2(\mu+\varepsilon\lambda)\leq \max\{M_2(\mu),M_2(\lambda)\}.
$$
By Lemma \ref{l:senoycontar}, both $M_1(\lambda)< C_{L_n}$ and $M_2(\lambda)<C_{L_n}$. Hence, by Theorem \ref{t:M1M2M3}, we would have $C_{\mu+\varepsilon\lambda}<C_{L_n}$. Again a contradiction.
\end{proof}

\begin{thm}\label{lessthan3}
For every $n\in\mathbb N$, we have that 
$$
C_{L_n}^0=1+2\cos\Big(\frac{\pi}{n+1}\Big)\le C_{L_n}<3,
$$ and, moreover,
$$
C_{L_n}^0=C_{L_n}\iff 2\le n\le8.
$$
\end{thm}

\begin{proof}
By Theorem~\ref{thm:C^0_Ln} we know already that $C_{L_n}^0=1+2\cos\big(\frac{\pi}{n+1}\big)$, $C_{L_n}^0\le C_{L_n}$ by definition, and, for the counting measure $\lambda$,  $C_{L_n}\le C_\lambda\le3$. If $C_{L_n} =3$, then $\lambda$ would be a minimizer. By Lemma~\ref{l:senoycontar}, we know that $M_1(\lambda)=2$. But, the previous Proposition~\ref{p:minimizer} tells us that, in this case, $M_1(\lambda)=C_{L_n} =3$, which is a contradiction.

The second claim is precisely Lemma~\ref{prevlemm}.
\end{proof}

\begin{thm}\label{t:mainDM}
For every $n\in\mathbb N$ there is $2\leq j\leq\lceil\frac{n}{2}\rceil$ and $k<\lceil\frac{n-2}{3}\rceil$ such that every doubling minimizer $\mu$ on $L_n$ satisfies
$$
\frac{\mu(B(j,1))}{\mu(B(j,0))}=\frac{\mu(B(1,2k+1))}{\mu(B(1,k))}=C_{\mu}. 
$$
\end{thm}

\begin{proof}
Suppose otherwise, for every $2\leq j\leq\lceil\frac{n}{2}\rceil$ there is $\mu_j$ on $L_n$ such that $C_{\mu_j}=C_{L_n}$ and
$$
\frac{\mu_j(B(j,1))}{\mu_j(B(j,0))}<C_{L_n}
$$
Let $\mu=\sum_{j=2}^{\lceil\frac{n}{2}\rceil} \mu_j$. It is clear from Proposition \ref{p:DMconvex} that $C_\mu=C_{L_n}$. However, by Lemma \ref{l:holder}, it follows that  $C_\mu^0<C_{L_n}$. Taking the symmetrization of $\mu$ this leads to a contradiction with Proposition \ref{p:minimizer}.

Similarly, if for every $k<\lceil\frac{n-2}{3}\rceil$ there is $\mu_k$ such that $C_{\mu_k}=C_{L_n}$ and
$$
\frac{\mu_k(B(1,2k+1))}{\mu_k(B(1,k))}<C_{L_n},
$$
then $\mu=\sum_k\mu_k$ satisfies $C_\mu=C_{L_n}$ but $M_1(\mu)<C_{L_n}$. This is a contradiction with Proposition \ref{p:minimizer}.
\end{proof}

 However, starting with an arbitrary doubling minimizer we next show  how to modify it to construct another one with somehow better properties:

\begin{prop}\label{p:minimizer2}
For every $n\in\mathbb N$ there is a symmetric measure $\mu$ on $L_n$ such that 
$$
C_{L_n}=C_\mu=\frac{\mu(B(\lceil\frac{n}{2}\rceil,1))}{\mu(B(\lceil\frac{n}{2}\rceil,0))}=\frac{\mu(B(1,2k+1))}{\mu(B(1,k))}=\frac{\mu(B(2,1))}{\mu(B(2,0))},
$$ 
where $k$ is given as in Theorem~\ref{t:mainDM}.
\end{prop}

\begin{proof}
Let $\mu$ be a symmetric measure on $L_n$ such that $C_\mu=C_{L_n}$. Suppose that $\frac{\mu(B(\lceil\frac{n}{2}\rceil,1)}{\mu(B(\lceil\frac{n}{2}\rceil,0))}<C_\mu$. Let $d_n=\frac{3+(-1)^n}{2}$ and set
$$
\varepsilon=\frac{C_\mu \mu(\lceil\frac{n}{2}\rceil)-\mu(B(\lceil\frac{n}{2}\rceil,1))}{C_\mu-d_n}.
$$
Let $\nu$ be the measure $\mu_{\lceil\frac{n}{2}\rceil,-\varepsilon}$ given as in \eqref{eq:perturbmu} (note that $\varepsilon<\mu(\lceil\frac{n}{2}\rceil)$).

The same reasoning as in the proof of Proposition \ref{t:minimizer} yields $C_\nu=C_\mu$. Moreover, by our choice of $\varepsilon$ we have
$$
\frac{\nu(B(\lceil\frac{n}{2}\rceil,1))}{\nu(B(\lceil\frac{n}{2}\rceil,0))}=\frac{\mu(B(\lceil\frac{n}{2}\rceil,1))-d_n\varepsilon}{\mu(\lceil\frac{n}{2}\rceil)-\varepsilon}=C_\mu=C_\nu.
$$ 
By Theorem \ref{t:mainDM}, we have
$$
C_\nu=\frac{\nu(B(\lceil\frac{n}{2}\rceil,1))}{\nu(B(\lceil\frac{n}{2}\rceil,0))}=\frac{\nu(B(1,2k+1))}{\nu(B(1,k))}.
$$
Now, suppose $\nu$ does not satisfy the requirements of the claim, that is $\frac{\nu(B(2,1))}{\nu(B(2,0))}<C_{L_n}$. Let 
$$
\theta=C_{L_n}\nu(2)-\nu(B(2,1)),
$$
and set $\tilde\nu=\nu_{1,\theta}$ as in \eqref{eq:perturbmu}. It is clear that
$$
\frac{\tilde \nu(B(2,1))}{\tilde \nu(B(2,0))}=\frac{\tilde \nu(B(\lceil\frac{n}{2}\rceil,1))}{\tilde \nu(B(\lceil\frac{n}{2}\rceil,0))}=C_{L_n}.
$$
It is also clear that $C_{\tilde \nu}^0\leq C_\nu^0$ and $M_1(\tilde \nu)<M_1(\nu)$. Finally, since $1\notin B(j,2k+1)$ for every $j,k$ appearing in the expression for $M_2(\tilde \nu)$, it follows that $M_2(\tilde \nu)=M_2(\nu)$. Hence, $C_{\tilde\nu}=C_\nu=C_{L_n}$, so $\tilde \nu$ satisfies the requirements in the claim.
\end{proof}

\section{Final remarks}\label{finalrem}

In this section we are going to recollect several remarks and open questions that we find of interest for possible future developments. Some of them are a natural consequence of the results we have already proved, and have been numerically checked   on a computer, up to a certain degree of accuracy.
\medskip

\begin{itemize}[leftmargin=*]

\item[--]
It would be interesting to find out the explicit relation between $n$ and $k$ in Theorem~\ref{t:mainDM}, as this would improve the lower bound $1+2\cos(\pi/(n+1))\leq C_{L_n}$. We have observed numerically that for $n\leq 200$ one essentially has $k=\lceil\frac{n-8.6}{5.6}\rceil$.
\medskip

\item[--]  Related to Theorem~\ref{t:mainDM} and Proposition~\ref{p:minimizer2} we have observed that, in all cases we have been able to explicitly calculate $C_{L_n}$, most of them with the help of a computer, the following holds:

\medskip\noindent
For every $n\ge2$, there exists a symmetric doubling minimizer $\mu\in DM(L_n)$ such that,
\medskip

\begin{enumerate}[leftmargin=*]
\item[(i)] There exists $k\in\mathbb N\cup\{0\}$ for which
$$
C_{L_n}=\frac{\mu(B(1,2k+1))}{\mu(B(1,k))}=\frac{\mu(B(n,2k+1))}{\mu(B(n,k))}.
$$

\item[(ii)] For every $j\in\{2,\dots,n-1\}$,
$$
C_{L_n}=\frac{\mu(B(j,1))}{\mu(B(j,0))}.
$$
\end{enumerate}

\medskip
\noindent
If this were the case, knowing $k$ as in the previous remark, it would a priori be possible to implement an algorithm to determine $C_{L_n}$ and calculate $\mu$ solving the following system of $\lceil\frac{n}{2}\rceil$ rational equations and unknowns (with the normalization $\mu(1)=1$ and writing $C=C_{L_n}$):
\medskip

$$
\begin{cases}
  C=\displaystyle\frac{1+\sum_{l=2}^{2k+2}\mu(l)}{1+\sum_{l=2}^{k+1}\mu(l)}   & \\[.5cm]
    C=\displaystyle\frac{1+\mu(2)+\mu(3)}{\mu(2)}   &\\[.5cm]
    C=\displaystyle\frac{\mu(2)+\mu(3)+\mu(4)}{\mu(3)}   & \\[.5cm]
...& \\[.5cm]
     C=\displaystyle\frac{\mu(\lceil\frac{n}{2}\rceil-1)+\mu(\lceil\frac{n}{2}\rceil)+\mu(\lceil\frac{n}{2}\rceil+1)}{\mu(\lceil\frac{n}{2}\rceil)}   .& 
\end{cases}
$$
\medskip

\noindent
In particular, we have been able to solve these equations for $n\le201$ (recall that by Proposition~\ref{lessthan3} we have that $C_{L_n}=1+2\cos\big(\frac{\pi}{n+1}\big)$, if $2\le n\le 8$). These are some examples:
\medskip

\noindent
$C_{L_9}\approx 2.9051661677540188$ is the largest root of the equation
$$
x^4-5 x^3+7 x^2-3 x+1=0.
$$

\medskip

\noindent
$C_{L_{10}}\approx 2.9229996101689726$ is the largest root of the equation
$$
x^4-3 x^3+x-1=0.
$$

\medskip

\noindent
$C_{L_{51}}\approx 2.9969167359239086$ is the largest root of the equation
\begin{align*}
&x^{26}-25 x^{25}+276 x^{24}-1747 x^{23}+6808 x^{22}-15708 x^{21}\\
&\qquad+14861 x^{20}+24091 x^{19}-92682 x^{18}+87057 x^{17}+77858 x^{16}\\
&\qquad
-234588 x^{15}+102327 x^{14}+199171 x^{13}-225057 x^{12}\\
&\qquad
-41798 x^{11}+165000 x^{10}-36531 x^9-58763 x^8+25759 x^7\\
&\qquad+10011 x^6-6268 x^5-646 x^4+597 x^3+6 x^2-12 x=0.
\end{align*}

\medskip

\item[--] Contrary to what happens for $\Z$ (Theorem~\ref{t:minimizersZ}), we do not know whether there is a unique doubling minimizer on $L_n$ (up to multiplicative constants). Observe that, as Example~\ref{examplenonunique} shows, uniqueness does not hold in $\N$.

\medskip

\item[--] Further extensions to more general graphs $G$ will be considered in the forthcoming paper \cite{DST}. In particular, the connection of the constant $C_G^0$ to spectral properties, the symmetries of doubling minimizers with respect to Aut$(G)$, the group of automorphisms of $G$, and the characterization of those graphs for which $C_G\le 3$.

\end{itemize}

\medskip

\noindent
\textbf{Acknowledgment.} The first author would like to thank Professor Miguel \'Angel Sama (UNED),  for his valuable help with the numerical calculations (and some enlightening conversations)  involved in the preliminary versions of this manuscript.

\end{document}